\documentclass[reqno]{amsart}
\UseRawInputEncoding
\makeatletter
\@namedef{subjclassname@2020}{\textup{2020} Mathematics Subject Classification}
\makeatother

\usepackage{amssymb}
\usepackage[all]{xy}
\usepackage{enumitem}
\usepackage{hyperref}

\def\commentout#1{{}}

\def\Z{{\mathbb Z}}

\def\N{{\mathbb N}}

\def\cC{{\mathcal C}}

\def\calI{{\mathcal I}}

\def\calT{{\mathcal T}}
\def\calM{{\mathcal M}}

\newcommand\im{\operatorname{im}} 
\newcommand\coim{\operatorname{coim}}

\newcommand\Hom{\operatorname{Hom}}

\newcommand\FinSets{{\operatorname{FinSets}}}

\newcommand\id{{\operatorname{id}}}

\newcommand\Isom{{{\operatorname{Isom}}}}
\newcommand\Extr{{{\operatorname{Extr}}}}
\newcommand\Mono{{{\operatorname{Mono}}}}

\newtheorem{theorem}{Theorem}[section]
\newtheorem{lemma}[theorem]{Lemma}
\newtheorem{proposition}[theorem]{Proposition}
\newtheorem{corollary}[theorem]{Corollary}

\theoremstyle{definition}
\newtheorem{definition}[theorem]{Definition}

\theoremstyle{remark}
\newtheorem{remark}[theorem]{Remark}

\numberwithin{equation}{section}

\usepackage{amsmath} 
\begin{document}
\title[Lov\`{a}sz's hom-counting theorem by inclusion-exclusion principle]
{Lov\`{a}sz's hom-counting theorem by inclusion-exclusion principle}

\author{Shoma Fujino}
\address{Mathematics Program \\ 
Graduate School of Advanced Science and Engineering\\
Hiroshima University, 739-8526 Japan}
\email{shomafujino0729@gmail.com}

\author{Makoto Matsumoto}
\address{Mathematics Program \\ 
Graduate School of Advanced Science and Engineering\\
Hiroshima University, 739-8526 Japan}
\email{m-mat@math.sci.hiroshima-u.ac.jp}

\thanks{
The second author is partially supported by JSPS
Grants-in-Aid for Scientific Research
JP26310211 and JP18K03213.
}

\date{\today}

\begin{abstract}
Let $\cC$ be the category of finite graphs.
Lov\`{a}sz (1967) shows that if $|\Hom(X,A)|=|\Hom(X,B)|$ 
holds for any $X$, then $A$ is isomorphic to $B$.
Pultr (1973) gives a categorical generalization using a similar argument.
Both proofs assume that each object has 
a finite number of isomorphism 
classes of subobjects.
Generalizations without this assumption are given by 
Dawar, Jakl, and Reggio (2021) and Reggio (2021). Here
another generalization without this assumption is given,
with a shorter proof.
Examples of categories are given, for which our theorem 
is applicable, but the existing theorems are not.
\end{abstract}

\maketitle
\section{Introduction}
In a category, it clearly holds that
$$
A \cong B \Rightarrow |\Hom(X,A)|=|\Hom(X,B)| \mbox{ for all objects } X. 
$$
A category where the converse holds is said to be combinatorial 
(Definition~\ref{def:combinatorial}). 
This notion is introduced by Pultr \cite{PULTR}
after the Lov\`{a}sz's memorial work \cite{LOVASZ-OPERATION},
and a considerable amount of studies exists:
a direct generalization of Lov\`{a}sz's proof for 
more general categories is given 
by Pultr \cite{PULTR}, a different approach by Isbell \cite{ISBELL}, 
and some new types of proofs are given by Dawar, Jakl, and Reggio \cite{DAWAR}
and Reggio \cite{REGGIO}.
This property and its generalization
in the category of graphs is widely studied, 
sometimes from computational aspects,
see Cai-Govorov \cite{CAI} and its references.
The aim of this paper
is to give yet another simple sufficient condition for a category 
to be combinatorial (Main Theorem~\ref{th:main}).
In the last section, we show some example categories,
to separate the scope of the existing theorems and ours.
There is a category to which our theorem is applicable, 
but the other theorems are not.

\section{Preliminary and Main Theorem}
The notions of mono, epi, pullback, pushout, and subobject
are standard, see MacLane \cite{MACLANE}. The term ``quotient object''
depends on the literature
(in Mitchell \cite[p.7]{MITCHELL} as a dual of subobject, 
and in Pultr \cite[Section~1.1]{PULTR} in a difference sense),
so here we use a less common word:
\begin{definition}\label{def:supobject}
A subobject of $A$ is a mono $m:B \to A$.
A supobject of $A$ is an epi $e:A \to C$.
\end{definition}

\begin{definition} \label{def:order}
Let $\cC$ be a category. For supobjects $q:X \to Q$ and
$q':X \to Q'$, we say $q\geq q'$ if there is an $h$
with $q'=h\circ q$. This gives a partial order
on the isomorphism classes of the supobjects
of $X$. The largest supobject
is the isomorphism class given by $\id_X$. 
If $q$ is proper (i.e., nonisomorphic, see Definition~\ref{def:extremal})
then the supobject $q:X \to Q$ is 
said to be proper. (This is equivalent to $q<\id_X$.)
A maximal supobject is
a supobject that is maximal among the proper subobjects.

Dual notions are similarly 
defined for subobjects of $X$.
To make clear, for subobjects $m:M \to X$
and $m':M' \to X$, $m\leq m'$ if $m=m'\circ h$ for some $h$.
The largest subobjects are isomorphic to $\id_X$.
\end{definition}

\begin{definition}(\cite[Definition~4.3.2]{BORCEUX})
\label{def:extremal}

An epimorphism $e:A \to B$ is an extremal epimorphism 
if $e=m\circ g$ where $m$ is mono, then $m$ is an isomorphism.
An epimorphism is a proper epimorphism, if it is not an 
isomorphism.
\end{definition}

\begin{definition}(\cite[p.12]{MITCHELL})
\label{def:image}

Let $f:X \to Y$ be a morphism. An image of $f$
is a subobject $m:\im f \to Y$ 
such that there is a $g:X \to \im f$ with $f=m\circ g$,
and if $f=m'\circ g'$ with another subobject $m':Z \to Y$,
then $m=m'\circ h$ for some $h:Z \to \im f $ (i.e. $m\leq m'$). 
Since $m$ is mono, $g'=h\circ g$ follows.

Dually, a coimage of $f$ is a supobject $e: X \to \coim f$,
such that there is an $g: \coim f \to Y$ with $f=g\circ e$,
and if $f=g'\circ e'$ with another supobject $e':X \to Z$,
then $e=h\circ e'$ for some $h:Z \to \coim f$ (i.e. $e\leq e'$).
\end{definition}

\begin{definition}
A category $\cC$ is locally finite, if for any objects $A,B$,
$\Hom(A,B)$ is a finite set.
\end{definition}
This terminology seems now common \cite{DAWAR}\cite{REGGIO},
but a different term ``quasifinite'' is used in Pultr \cite{PULTR}.
The following notion {\em combinatorial}
is the theme of this paper.
\begin{definition} (\cite[1.7~Definition]{PULTR})
\label{def:combinatorial}

A locally finite category $\cC$ is said to be
combinatorial, if for all objects $X$
$$
|\Hom(X,A)|=|\Hom(X,B)|
$$
hold then $A$ is isomorphic to $B$.
\end{definition}
Lov\`{a}sz \cite{LOVASZ-OPERATION} 
proved that the categories of operations with finite structures
(including the category of finite graphs) are combinatorial.
Pultr gives a categorical generalization, using a similar argument.
See Theorem~\ref{th:pultr}.

In the rest of this section, we shall show another 
sufficient condition for a category to be combinatorial
(our main Theorem~\ref{th:main}). We start from some preliminary.

\begin{lemma}\label{lem:mono} (Lov\`{a}sz\cite[Lemma~1]{LOVASZ-DIRECT})

Let $\cC$ be a locally finite category. 
If there are monomorphisms $m:A \to B$ and $n:B \to A$,
then $m$ and $n$ are isomorphisms. 
Dually, if there are epimorphisms $e:A \to B$ and
$f:B \to A$, then $e$ and $f$ are isomorphisms. 
\end{lemma}
\begin{proof}
We prove only the dual. Since $|\Hom(B,B)|$
is finite, the compositions $(ef)^n$ for $n \in \N$
must coincide for different $n$, say, for $n$ and $n+m$ with $m\geq 1$.
Since $(ef)^n$ is epi, $(ef)^n=(ef)^m(ef)^{n}$ implies
$(ef)^m=\id_B$, and hence $f$ is a splitting monomorphism: putting $g:=(ef)^{m-1}e$,
$gf=\id_B$. Thus $fgf=f$, and since $f$ is epi, $fg=\id_A$.
\end{proof}

\begin{lemma}\label{lem:pushout}
Let $\cC$ be a category. Let $q_i:X \to Q_i$, $i=1,2$, be supobjects
(Definition~\ref{def:supobject}).
Suppose that $q_1, q_2$ has a pushout $q_{3}:X \to Q_3$.
Let $Y$ be an object. Then, inside $\Hom(X,Y)$, we have
$$
q_3^*\Hom(Q_3, Y)=q_{1}^*\Hom(Q_1,Y)\cap q_2^*\Hom(Q_2, Y)
$$
holds, where 
$$
q_i^*\Hom(Q_i,X)=\{f\circ q_i \in \Hom(X,Y) \mid f \in \Hom(Q_i,X)\}.
$$ 
\end{lemma}
\begin{proof}
This follows from the definition of the pushout:
by Yoneda functor $\Hom(-,Y)$, 
a pushout is mapped to a pullback.
Since $q_i$ $(i=1,2)$ are
epi, the morphisms between $\Hom$'s are injective,
and the pullback is isomorphic to the intersection.
\end{proof}

Let us denote by $\Mono(A,B)$ the set of monomorphisms
between $A$ and $B$.
\begin{definition}\label{def:calI} 
Let $\cC$ be a category. 
Let $\calI$ denote a subclass of monomorphisms,
including the identities.
The set of $\calI$-monomorphism from $A$ to $B$ is denoted by
$$
\Mono_\calI(A,B) \subset \Mono(A,B).
$$
\end{definition}

\begin{definition}\label{def:calM}
For each $X$, we specify a subclass {\em $\calM$-supobject}
of the class of the supobjects $q:X \to Q$ .
\end{definition}
For most applications considered, $\calM$ equals
to the class of maximal epimorphisms (see Definition~\ref{def:order}),
and $\calI$ equals to the class of monomorphisms.
We introduced these notions, to make
the condition of Main Theorem~\ref{th:main}
as weak as possible.

We state the main theorem of this paper.
\begin{theorem}\label{th:main} (Main Theorem)
Let $\calM$ and $\calI$ be as in Definitions~\ref{def:calM},
\ref{def:calI}.
Let $\cC$ be a locally finite category satisfying the following
conditions.
\begin{enumerate}
\item For any finite number of $\calM$-supobjects $e_i:X \to Q_i$, $i=1,2,\ldots,m$,
there exists a pushout. \label{enum:pushout}
\item For any object, the set of isomorphism classes 
of its $\calM$-supobjects is finite. \label{enum:qfin}
\item If $f:X \to Y$ is not $\calI$-mono, it factors through \label{enum:mono}
an $\calM$-supobject $X \to Q$.
\item If $f:X \to Y$ factors through an $\calM$-supobject $X \to Q$,
then $f$ is not $\calI$-mono. \label{enum:factor}
\end{enumerate}

Then, $\cC$ is combinatorial (Definition~\ref{def:combinatorial}).
\end{theorem}
\begin{proof}
Suppose that 
$$
|\Hom(X,A)|=|\Hom(X,B)|
$$
holds for any $X$. Let $Z$ be an arbitrary object.
Let $q_i:Z \to Q_i$ $(i=1,2,\ldots,m)$ be the representatives of the
$\calM$-supobjects of $Z$ (they are finite by Condition~\ref{enum:qfin}).
Take $f\in \Hom(Z,A)$. By Conditions~\ref{enum:mono} and \ref{enum:factor}, 
$f$ is not $\calI$-mono if and only if $f$ factors through one of $Q_i$.
Thus, we have
$$
\Mono_\calI(Z,A)=\Hom(Z,A) \setminus \bigcup_{i=1}^m q_i^*\Hom(Q_i,A).
$$
Now we use Lemma~\ref{lem:pushout} and the inclusion-exclusion principle
to obtain
\begin{eqnarray}\label{eq:inclusion-exclusion}
|\Mono_\calI(Z,A)| 
& = & |\Hom(Z,A)| -\sum_{1\leq i \leq m}|q_i^*\Hom(Q_i,A)| \nonumber\\
& & +\sum_{1\leq i < j \leq m}|q_i^*\Hom(Q_i,A)\cap q_j^*\Hom(Q_j,A)| \nonumber\\
& & -\sum_{1\leq i<j<k \leq m}
|q_i^*\Hom(Q_i,A)\cap q_j^*\Hom(Q_j,A)\cap q_k^*\Hom(Q_k,A)|\nonumber\\
& & +\cdots \nonumber\\
&=& |\Hom(Z,A)| -\sum_{1\leq i \leq m}|q_i^*\Hom(Q_i,A)| \nonumber\\
& & +\sum_{1\leq i < j \leq m}|q_{ij}^*\Hom(Q_i\stackrel{Z}{\coprod}Q_j,A) 
\nonumber\\
& & -\sum_{1\leq i<j<k \leq m}
|q_{ijk}^*\Hom(Q_i\stackrel{Z}{\coprod}Q_j \stackrel{Z}{\coprod}Q_k,A)+\cdots,
\end{eqnarray}
where 
$$
q_{ij}:Z \to Q_i\stackrel{Z}{\coprod}Q_j
$$
denotes the pushout of $q_i$ and $q_j$, 
$$
q_{ijk}:Z \to Q_i\stackrel{Z}{\coprod}Q_j\stackrel{Z}{\coprod}Q_k
$$
denotes the pushout of $q_i$, $q_j$, $q_k$, and so on. 
Since the expression (\ref{eq:inclusion-exclusion}) is 
given by a combination of $|\Hom(-,A)|$, we have the same value
when $A$ is replaced with $B$. Namely,
$$
|\Mono_\calI(Z,A)|=|\Mono_\calI(Z,B)|.
$$
If we put $Z=A$, the left-hand side contains $\id_A$,
hence there is a monomorphism $A\to B$.
The symmetric argument gives a monomorphism $B \to A$,
and Lov\`{a}sz's Lemma~\ref{lem:mono} completes the proof.
\end{proof}
Often, the following conditions are (stronger but) easier
to check.
\begin{theorem}~\label{th:weak} $ $

Let $\cC$ be a locally finite category
satisfying the following conditions.
\begin{enumerate}
\item For any epimorphisms $e_i:X \to Q_i$, $i=1,2$,
there exists a pushout. \label{enum:pushout2}
\item For any object, the set of isomorphism classes 
of its maximal supobjects is finite. \label{enum:qfin2}
\item For any proper supobject $q:X \to Q$, there
is a maximal supobject $q':X \to Q'$
such that $q'\geq q$
(see Definition~\ref{def:order} for the terminology). \label{enum:maximal2}
\item If $f:X \to Y$ is not mono, it factors through \label{enum:mono2}
a proper supobject $X \to Q$.
\item If $f:X \to Y$ factors through a proper supobject $X \to Q$,
then $f$ is not mono. \label{enum:factor2}
\end{enumerate}
Then, $\cC$ is combinatorial.
\end{theorem}
\begin{proof}
This is obtained from the above Theorem~\ref{th:main}
by considering the case where $\calM$-supbojects are
the proper supobjects (Definition~\ref{def:order})
and $\calI$ is the class of monomorphisms.
Note that the last three conditions imply the last two
conditions in Theorem~\ref{th:main}. Indeed,
if $f$ is not-mono, it factors through a proper supobject,
and then through a maximal supobject. Conversely, if $f$ factors
through a maximal supobject (one of the proper supobjects), 
then $f$ is not mono.
\end{proof}
An even weaker form is the following.
\begin{corollary}\label{cor:weak}
Among the five conditions in 
Theorem~\ref{th:weak}, we replace Conditions~\ref{enum:mono2} 
and \ref{enum:factor2} with

(\ref{enum:mono2}') Any morphism has a coimage.

(\ref{enum:factor2}') $f:X \to Y$ is mono if and only if
$X \to \coim f$ is an isomorphism.

\noindent
Under these five conditions, $\cC$ is combinatorial.
\end{corollary}
\begin{proof}
Suppose these conditions. 
If $f:X \to Y$ is not mono, then $X \to \coim f$ is a proper
supobject by (\ref{enum:factor2}'), which implies (\ref{enum:mono2}).
If $f:X \to Y$ factors through a proper subobject $X \to Q$,
then $X \to Q \to \coim(f)$ given by the universality 
of the coimage is not an isomorphism 
(since if isomorphic, then $X \to Q$ is a splitting monomorphism
and epimorphism, thus an isomorphism, contradicting to the assumption).
Hence $f$
is not mono by (\ref{enum:factor2}'), which implies (\ref{enum:factor2}).
\end{proof}
We remark the following, related to (\ref{enum:factor2}').
\begin{lemma}
Let $\cC$ be a category with coimages.
The followings are equivalent.
\begin{enumerate}
\item For any $f:X \to Y$, 
if $X \to \coim(f)$ is an isomorphism, then $f$ is mono.
\item For any $f:X \to Y$, $\coim(f) \to Y$ is mono.
\end{enumerate}
\end{lemma}
\begin{proof}
The second condition implies the first. For the converse,
let $g:\coim(f) \to Y$. Take $\coim(f) \to \coim(g)$. This is a morphism of
supobjects of $X$, and $\coim(f) \geq \coim(g)$. 
By the universality of $\coim(f)$, we have a morphism
$\coim(g) \to \coim(f)$ of supobjects of $X$. Hence the converse 
inequality holds. Consequently, 
these are isomorphisms, and by the first condition,
$g$ is mono.
\end{proof}

\begin{remark}\label{rem:fingp}
One can show that the categories of finite undirected graphs,
finite directed graphs, the functor-categories $\FinSets^\cC$
from a finite category $\cC$ to the category of
finite sets and a category of finite groups satisfy the
conditions of Corollary~\ref{cor:weak}
(hence those of the main Theorem~\ref{th:main}),
hence are combinatorial.
(Pultr's theorem is also applicable for these examples.)

To see the strongness of such a statement, consider
the category of finite groups. 
For finite groups $A,B,C$, suppose that 
$A\times B \cong A \times C$ holds. Then, 
for any $Z$, 
\begin{eqnarray*}
|\Hom(Z,A)||\Hom(Z,B)|
&=&|\Hom(Z,A\times B)| =|\Hom(Z,A\times C)| \\
&=&|\Hom(Z,A)||\Hom(Z,C)|,
\end{eqnarray*}
and since $|\Hom(Z,A)|\geq 1$, we have
$$
|\Hom(Z,B)|=|\Hom(Z,C)|,
$$
and combinatoriality implies that $B\cong C$.
This is non-trivial, see for example \cite{CANCELLATION}.
This property is well-studied as a cancellation law,
see Lov\`{a}sz \cite{LOVASZ-CANCELLATION}.

Note also that Dawar-Jakl-Reggio's
generalization \cite[Theorem~5]{DAWAR} 
(see Theorem~\ref{th:dawar} below)
can not be applied to finite groups, since finite groups 
have pushouts for epimorphisms, but the pushout of $\Z/2$
and $\Z/3$ under the trivial group does not exist
(it is known that the pushout is 
$PSL(2,\Z)$ in the category of groups).
Reggio's generalization (Theorem~\ref{th:reggio} below)
works, see \cite[Example~4.6]{REGGIO}.
\end{remark}

\section{A variant of Pultr's result for a comparison}
The following theorem is a slight 
generalization of Pultra's theorem
\cite[2.2~Theorem]{PULTR}
which is based on an argument given by 
Lov\`{a}sz \cite[(5), p.326]{LOVASZ-OPERATION}. 
We include a proof of this, mainly to show 
the difference from the proof of 
Main Theorem~\ref{th:main},
and partly because the statement is slightly 
stronger, and to give a variant (Theorem~\ref{th:var} below).

\begin{theorem} \label{th:pultr}

Let $\cC$ be a locally finite category satisfying
the following conditions.
\begin{enumerate}
\item For each object, the isomorphism classes of 
its subobjects is finite. \label{enum:fin}
\item Every morphism $f$ has its image. \label{enum:image}
\item \label{enum:extremal} 
If $h=g\circ f$ and $g$ is an image of $h$,
then $f$ is an extremal epimorphism.
\end{enumerate}
Then, $\cC$ is combinatorial. 
\end{theorem}
We shall give a proof soon.
The above theorem is slightly stronger than the 
following original Pultr's theorem.
Note that a quasifinite category in Pultr's terminology
is a locally finite category in our terminology.
\begin{theorem} (Pultr \cite[2.2~Theorem]{PULTR}) \label{th:original}
Let $\cC$ be a locally finite category satisfying
the conditions (\ref{enum:fin}), (\ref{enum:image})
in the above Theorem~\ref{th:pultr}, and

(\ref{enum:extremal}') Every quotient is an epimorphism.

\noindent
Then, $\cC$ is combinatorial. 
\end{theorem}
We shall define the term {\em quotient} now,
but use only in the rest of this section to avoid confusions.
A morphism $e:A \to B$
is a quotient, if in $e=\id_B\circ e$, $\id_B$ is an image of $e$
\cite[1.1~Definition]{PULTR}.
Theorem~\ref{th:pultr} implies Theorem~\ref{th:original}
as follows. 
\begin{proof}
It suffices to show that if every quotient 
is an epimorphism, then the third condition of 
Theorem~\ref{th:pultr} follows.
Suppose that $f=g \circ h$ with $g$ being an image of $f$.
Then $h$ is a quotient \cite[1.4~1)]{PULTR}.
By the third condition,
$h$ is a quotient and epi, which implies that $h$ is 
an extremal epimorphism
\cite[1.2 Remark]{PULTR}.
\end{proof}
For a comparison to our main result Theorem~\ref{th:main},
we would like to give a proof of Theorem~\ref{th:pultr},
which is very similar to those by Lov\`{a}sz and Pultr.
\begin{proof} (of Theorem~\ref{th:pultr}).
Let $\calT$ be a system of objects of $\cC$ containing
exactly one representative from each isomorphism class.
We construct a mapping
\begin{equation}\label{eq:classify}
\Hom(A,B) \to \calT 
\end{equation}
by mapping $h$ to the class of $\im h$. 
Thus if $h=g \circ f$ with $g$ being the image of $h$ with domain $T$
(one can choose a unique $T$ and $g$ by the uniqueness of the
image and by the definition of $\calT$),
then $h$ is mapped to $T$.
This gives a disjoint decomposition
\begin{equation}\label{eq:disjoint}
\Hom(A,B) = \coprod_{T \in \calT} \Hom(A,B)_T, 
\end{equation}
where $\Hom(A,B)_T$ denotes the inverse image of $T$.
Since $\cC$ is locally finite, this is a finite 
sum. 
We consider a mapping for $T$ appeared in the sum 
(\ref{eq:disjoint})
\begin{equation}\label{eq:composite}
\Extr(A,T)\times \Mono(T,B) \to \Hom(A,B)_T, \quad (f,g) \mapsto g\circ f,
\end{equation}
where $\Extr(A,T)$ means the set of extremal epimorphisms.
Take an $f\in \Extr(A,T)$,
and a $g\in \Mono(T,B)$.
Let $g'$ be an image of $g\circ f$.
Thus we have $g \circ f=g'\circ f'$.
The universality of the image implies that there is a mono $m$
with $g\circ m=g'$ and $f=m\circ f'$ (Definition~\ref{def:image}).
Since $f$ 
is an extremal epi, $m$ is isomorphic. This implies that
$g$ is an image of $g \circ f$, hence
(\ref{eq:composite}) is well-defined. It is surjective
since if $h \in \Hom(A,B)_T$, then there are $f,g$ with
$h=g\circ f$ such that $g$ is an image of $h$ with domain $T$
by the comment after (\ref{eq:classify}). Then
$f$ is an extremal epimorphism by Condition~(\ref{enum:extremal})
which means the surjectivity of (\ref{eq:composite}).
The fiber, i.e., the inverse image of one point in (\ref{eq:composite})
has the same cardinality as $\Isom(T,T)$, since the group acts
faithfully and transitively on the fiber, as follows. 
Fix an $(f,g)$. Then $(f',g')$ maps to the same element in 
$\Hom(A,B)_T$ if and only if $g\circ f=g'\circ f'$. Let $h$
be this composition. 
By the above argument, $g$ and $g'$ are images of $h$,
and it follows that $g'=g\circ m$ for an isomorphism 
$m \in \Isom(T,T)$.
Thus $\Isom(T,T)$ transitively acts on the fiber
by $(f,g) \mapsto (m^{-1}f,gm)$. 
Since $g$ is mono, such an $m$ is unique, which shows the 
faithfulness. Thus (\ref{eq:composite}) implies
\begin{equation}
|\Extr(A,T)|\times |\Mono(T,B)| 
= |\Isom(T,T)| \times |\Hom(A,B)_T| 
\end{equation}
and hence
\begin{equation}\label{eq:main}
|\Hom(A,B)|=\sum_{T\in \calT}
|\Isom(T,T)|^{-1} |\Extr(A,T)|\times |\Mono(T,B)|. 
\end{equation}
We claim that 
$$
|\Hom(T,B)|=|\Hom(T,C)| \mbox{ for all $T\in \calT$}
$$
implies that
$$
|\Mono(T,B)|=|\Mono(T,C)| \mbox{ for all $T\in \calT$}.
$$
Then, by putting $T=B$ there is a monomorphism $B\to C$,
and a symmetric argument gives a monomorphism $C \to B$,
and Lemma~\ref{lem:mono} completes the proof.
Let us prove the claim.
From $|\Hom(T,B)|=|\Hom(T,C)|$,
(\ref{eq:main}) implies
$$
0=
\sum_{T'\neq T}
\frac{|\Extr(T,T')|}{|\Isom(T',T')|}(|\Mono(T',B)|-|\Mono(T',C))
+(|\Mono(T,B)|-|\Mono(T,C)|).
$$
(Since $\Extr(T,T)=\Isom(T,T)$, by Lemma~\ref{lem:mono} for epi.)
This implies that if
$$
|\Mono(T,B)| \neq |\Mono(T,C)|
$$
then 
\begin{equation}\label{eq:monot}
|\Mono(T_1,B)| \neq |\Mono(T_1,C)| 
\end{equation}
for some $T_1$ with a non-isomorphic extremal epimorphism $e:T\to T_1$.
By the assumption, 
$$
|\Hom(T_1,B)| = |\Hom(T_1,C)|
$$
holds, and we may iterate the same argument for $T_1$, to have $T_2$
with a proper extremal epimorphism $e_1:T_1 \to T_2$ 
(For proper, see Definition~\ref{def:extremal}). 
In this way, we have an infinite sequence of objects $e_i:T_i \to T_{i+1}$.
They are mutually non-isomorphic.
(Assume any two are isomorphic, say $T_i$ and $T_j$, $i<j$.
Then there is an isomorphism, hence an epimorphism $T_j \to T_i$,
and Lemma~\ref{lem:mono} implies that the epimorphism $T_i \to T_j$ is an 
isomorphism, which implies $e_i$ is 
a monomorphism, and an extremal epimorphism,
and thus
$e_i$ is an isomorphism, leading to
a contradiction.) 
These $T_i$'s are subobjects of $B$ or $C$,
since in (\ref{eq:monot}) one of the two is not empty, 
so one of $B$ and $C$ has infinitely many non-isomorphic subobjects 
$T_i$. 
This contradicts the finiteness 
in Condition~\ref{enum:fin}.
\end{proof}
Recall that a preordered set is
well-founded, if every non-empty subset
has a minimal element. From the proof above, the following 
variant holds, which is our second main result.
\begin{theorem}\label{th:var}
Theorem~\ref{th:pultr} also holds if the condition (\ref{enum:fin})
is replaced with:

(\ref{enum:fin}') For each object, the preordered class
of its supobjects is well-founded.
\end{theorem}
\begin{proof}
If $\cC$ is not combinatorial, then
in the proof, $T_i$ gives an infinite sequence of strictly decreasing
supobjects of $T$, contradicting (\ref{enum:fin}').
\end{proof}

We remark that Dawar, Jakl, Reggio \cite[Theorem~5]{DAWAR} gives a
different sufficient condition, as follows.
\begin{theorem} (Dawar-Jakl-Reggio) \label{th:dawar}
Let $\cC$ be a locally finite category. If $\cC$
has pushouts and a proper factorization system, then it is combinatorial.
\end{theorem}
Reggio gives another sufficient condition \cite[Theorem~4.3]{REGGIO}
\begin{theorem} (Reggio) \label{th:reggio}
Let $\cC$ be a locally finite category. If $\cC$
has a proper factorization system $(Q,M)$
such that $\cC$ is $Q$-well-founded. Then it is combinatorial.
\end{theorem}
We don't describe the notion of proper factorization systems here,
see \cite[Appendix~A]{REGGIO}.

\section{Examples of categories to separate the scope of theorems}
Here we construct a category, for which
a weaker form of Main Theorem~\ref{th:weak} (and consequently
Main Theorem~\ref{th:main}) and Reggio's
Theorem~\ref{th:reggio}
are applicable, but Pultr's Theorem~\ref{th:pultr} (or
the method by Lov\`{a}sz), its variant Theorem~\ref{th:var}
and Dawar-Jakl-Reggio's Theorem~\ref{th:dawar}
are not applicable.

\begin{definition}\label{def:ab}
We define a category $\cC$ as follows.
\begin{enumerate}
\item Objects are $P_i$ for all integer $i$.
\item $\Hom(P_i,P_i)=\{\id_{P_i}\}$.
\item $\Hom(P_i,P_j)=\{a,b\}$ for $i>j$, $a\neq b$.
\item $\Hom(P_i,P_j)=\emptyset$ for $i<j$.
\end{enumerate}
Composition lows are
$$
aa=a,\
ba=a,\
ab=b,\
bb=b.
$$
\end{definition}
It is easy to check that this forms a category.
\begin{lemma}
Every non-identity morphism in $\cC$ is mono but not epi.
\end{lemma}
\begin{proof}
The equality $aa=ba$ implies $a$ is not epi, and
so is $b$ since $ab=bb$. On the other hand, 
$aa\neq ab$ implies $a$ is mono, as well as
$b$ by $ba \neq bb$.
\end{proof}
\begin{lemma}
The above category $\cC$ satisfies the five conditions
in the weak form of Main Theorem~\ref{th:weak}, and hence is combinatorial.
\end{lemma}
\begin{proof}
Clearly, $\cC$ is locally finite. 
In Theorem~\ref{th:weak},
(\ref{enum:pushout2}) follows since there is no epimorphism
but the identities. Thus there are no proper supobects, 
no maximal supobjects.
Thus (\ref{enum:qfin2}) and (\ref{enum:maximal2}) are satisfied. 
Since every morphisms are mono, (\ref{enum:mono2}) follows.
Since there is no proper supobject, (\ref{enum:factor2}) holds.
\end{proof}

\begin{lemma}
In the above $\cC$, every object has infinitely many
isomorphism classes of subobjects. Thus, Pultr's 
Theorem~\ref{th:pultr} (i.e. Lov\`{a}sz's argument \cite{LOVASZ-OPERATION})
will not work. The supobjects of each object are not well-founded, and thus
the variant Theorem~\ref{th:var} is neither applicable.

Dawar-Jakl-Reggio's Theorem~\ref{th:dawar}
can not be applied to $\cC$, because they 
have no pushout for $a,b$ (no commutative squares for $a,b$
with the same domain).

Reggio's Theorem~\ref{th:reggio}
can be applied to $\cC$, with $(Q,M)$ 
being $(\{\id\}, \{a,b,\id\})$. 
\end{lemma}

\begin{proposition}
We shall see the ordered set $\Z$ 
as a category, with objects integers, 
and $\Hom(i,j)$ is a singleton (an emptyset respectively)
if $i\geq j$ ($i<j$ respectively).
Then, Main Theorem~\ref{th:main} is applicable,
but Pultr's Theorem~\ref{th:pultr}, its variant
Theorem~\ref{th:var}, Dawar-Jakl-Reggio's Theorem~\ref{th:dawar},
Reggio's Theorem~\ref{th:reggio} are not applicable.
\end{proposition}
\begin{proof}
Every morphisms are mono and epi. 
For Main Theorem~\ref{th:main}, we take
$\calM$-supobjects to be empty, and $\calI$ to be the set of
all morphisms.
Then all the conditions in Theorem~\ref{th:main} are satisfied.

There are infinitely many
non-isomorphic subobjects for each object, and hence Pultr's theorem
can not be applied. Since the supobjects have no minimal elements,
it variant Theorem~\ref{th:var} can not be applied. 
Since there is a morphism 
which is mono and epi but not isomorphic, 
\cite[Lemma~A.2 (b)]{REGGIO} shows that
this category has no proper factorization system, 
and hence Dawar-Jakl-Reggio's Theorem and
Reggio's Theorem are not applicable.
\end{proof}
\begin{proposition}
See the ordered set $\N$ as a category as above. 
Then, Theorem~\ref{th:pultr} is not applicable, but 
its variant Theorem~\ref{th:var} is applicable.
\end{proposition}
This is because $\N$ has infinitely many subobjects,
but finitely many supobjects. A converse statement
holds for the converse ordered set $-\N$.
We summarize the applicabilities in Table~\ref{tab:scope}. For the 
ordered sets, Main Theorem~\ref{th:main} is
applied with $\calM$ empty and $\calI$ all the morphisms.
\begin{table}[hbtp]
\caption{Scope of the theorems for combinatoriality}
\label{tab:scope}
\centering
\begin{tabular}{l|cccccc}
\hline
& Main & Weak & Pultr & Its variant & Dawar-et.al. & Reggio \\
& Th.\ref{th:main} 
&Th.\ref{th:weak}
&Th.\ref{th:pultr}
&Th.\ref{th:var}
&Th.\ref{th:dawar}
&Th.\ref{th:reggio} \\
\hline \hline
digraphs & yes & yes & yes & yes & yes& yes \\
finite groups & yes & yes & yes & yes & no & yes \\
Definition~\ref{def:ab}
& yes & yes & no & no & no & yes \\
$\Z$ & yes & no & no & no & no & no \\
$\N$ & yes & no & no & yes & no & no \\
$-\N$ & yes & no & yes & no & no & no \\
\hline
\end{tabular}
\end{table}

\bibliographystyle{plain}
\bibliography{sfmt-kanren-cat}

\begin{thebibliography}{10}

\bibitem{BORCEUX}
Francis Borceux.
\newblock {\em Handbook of Categorical Algebra I: Basic Category Theory},
  volume~50 of {\em Encyclopedia of Mathematics and its Applications}.
\newblock Cambridge University Press, 1994.

\bibitem{CAI}
J.Y. Cai and A.~Govorov.
\newblock On a theorem of {Lov\'{a}sz} that hom(.,h) determines the isomorphism
  type of h.
\newblock In {\em 11th Innovations in Theoretical Computer Science Conference
  (ITCS 2020)}, pages 17:1--17:15. Dagstuhl Publishing, 2021.

\bibitem{DAWAR}
A.~Dawar, T.~Jakl, and L.~Reggio.
\newblock {Lov\'{a}sz}-type theorems and game comonads.
\newblock In {\em Proceedings of the 36th Annual ACM/IEEE Symposium on Logic in
  Computer Science}, LICS, pages 1--13. IEEE, 2021.

\bibitem{CANCELLATION}
R.~Hirshon.
\newblock On cancellation in groups.
\newblock {\em The American Mathematical Monthly}, 76:1037--1039, 1969.

\bibitem{ISBELL}
J.~Isbell.
\newblock Some inequalities in hom sets.
\newblock {\em Journal of Pure and Applied Algebra}, 76:87--110, 1991.

\bibitem{LOVASZ-OPERATION}
L.~Lov\`{a}sz.
\newblock Operations with structures.
\newblock {\em Acta Math. Acad. Sci. Hungar.}, 18:321--328, 1967.

\bibitem{LOVASZ-CANCELLATION}
L.~Lov\`{a}sz.
\newblock On the cancellation law among finite relational structures.
\newblock {\em Period. Math. Hungar.}, 1:145--156, 1971.

\bibitem{LOVASZ-DIRECT}
L.~Lov\`{a}sz.
\newblock Direct product in locally finite categories.
\newblock {\em Acta Scientiarum Mathematicarum}, 33:319--322, 1972.

\bibitem{MACLANE}
S.~MacLane.
\newblock {\em Categories for the Working Mathematician}, volume~5 of {\em
  Graduate Texts in Mathematics}.
\newblock Springer-Verlag, New York, Heidelberg, Berlin, 2nd edition, 1998.

\bibitem{MITCHELL}
B.~Mitchell.
\newblock {\em Theory of categories}, volume~17 of {\em Pure and Applied
  Mathematics}.
\newblock Academic Press, New York and London, 1965.

\bibitem{PULTR}
A.~Pultr.
\newblock Isomorphism types of objects in categories determined by numbers of
  morphisms.
\newblock {\em Acta Scientiarum Mathematicarum}, 35:155--160, 1973.

\bibitem{REGGIO}
L.~Reggio.
\newblock Polyadic sets and homomorphism counting.
\newblock arxiv:2110.11061, 2021.

\end{thebibliography}


\end{document}